\documentclass[letterpaper,12pt]{article}
\usepackage{color}
\usepackage{hyperref}
\usepackage{graphicx}
\usepackage{psfrag}
\usepackage{epsf}
\usepackage{amsmath,amsfonts,amssymb,latexsym}
\usepackage{enumerate}
\usepackage{cite}
\usepackage{tikz}
\usetikzlibrary{decorations.pathmorphing}
\usetikzlibrary{decorations.markings}
\usetikzlibrary{shapes,positioning}
\usepackage{tkz-graph}
\usepackage{setspace}

\def\fakeend{\end{document}}
\newcommand{\capwidth}{.7\textwidth}

\newcommand{\ignore}[1]{}

\newtheorem{theorem}{Theorem}[section]
\newtheorem{corollary}[theorem]
{Corollary}
\newtheorem{lemma}[theorem]{Lemma}

\newtheorem{definition}[theorem]{Definition}

\newtheorem{observation}[theorem]{Observation}

\date{}

\newenvironment{proof}%
{\noindent{\bf Proof.}\ }%
{\hfill\eopf\par\bigskip}%

\def\i4c{{internally--4--connected}}
\def\2cc{{2--crossing--critical}}

\def\m2{{{\cal M}_2^3}}

\newcommand{\eopf}{\raisebox{0.8ex}{\framebox{}}}

\title{Intersections of Cycling 2-factors}

\author{Drew J. Lipman \\ Clemson University}

\begin{document}

\maketitle
\thispagestyle{empty}
\pagestyle{empty}

\abstract Define an embedding of graph $G=(V,E)$ with $V$ a finite set of distinct points on the unit circle and $E$ the set of line segments connecting the points.
Let $V_1,\ldots,V_k$ be a labeled partition of $V$ into equal parts.
A 2-factor is said to be {\em cycling} if for each $u\in V$, $u\in V_i$ implies $u$ is adjacent to a vertex in $V_{i+1\: (mod \: k)}$ and a vertex in $V_{i-1\: (mod\: k)}$.
In this paper, we will present some new results about cycling 2-factors including a tight upper bound on the minimum number of intersections of a cycling 2-factor for $k=3$.

\section{Introduction}
Plane embeddings of graphs have a long history in graph theory.
In this paper, we study geometric graphs with vertices in convex position.
%Definition of Geometric graph
\begin{definition}
A graph $G=(V,E)$ is a {\em geometric graph} if the vertices are a set of distinct points in the plane and the edges are a set of straight-line segments.
We say two edges {\em intersect} if the straight-line segments intersect. 
\end{definition}

In the language of \cite[p.~88-90]{Diestel}, a geometric graph can be thought of as a {\em plane graph} where each {\em arc} is a trivial polygonal curve.

%History with 2 color points in the plane
Problems with geometric graphs have been very well studied for graphs with a bipartition of the vertices $V=V_1\cup V_2$ with $|V_1|=|V_2|$, see \cite{RBSurvey}.
One of these problems is the {\em Alternating Hamiltonian Cycle} problem, that is to find a Hamiltonian cycle that alternates between $X$ and $Y$ and has the minimum number of intersections between edges.
In \cite{KKY2000}, Kaneko, Kano, and Yoshimoto showed that the minimum number of intersections of an alternating Hamiltonian cycle is bounded above by $|V_1|-1$, which is tight in some instances.

We consider the case when $V$ is partitioned into at least three equal sets and the points are in {\em convex position}.
%Definition of convex position
\begin{definition}
Given a set of distinct points $\{x_1,\ldots,x_n\}\subseteq \mathbb{R}^2$ we say the points are in {\em convex position} if none of the points can be represented as a convex combination of the others.
That is,

$x_i\notin conv(\{x_1,\ldots,x_n\}\setminus\{x_i\})$ for all $i=1,2,\ldots,n$.
\end{definition}

Notice that, given a set of points in convex position, any subset of these points will also be in convex position.

Geometric graphs with points in convex position can be thought of as points equally spaced around a circle with chords connecting the points in different sets of the partition of $V$.

%Definition of cycling 2-factors
\begin{definition}
Let $G=(V,E)$ be a graph.
Let $V_1,\ldots,V_k$ be a labeled partition of $V$ into equal parts, $k\geq 3$.
A 2-factor of $G$ is said to be {\em cycling} if, for each $u\in V$, $u\in V_i$ implies $u$ is adjacent to a vertex in $V_{i+1\: (mod \: k)}$ and a vertex in $V_{i-1\: (mod\: k)}$.
\end{definition}

This is a generalization of the {\em alternating} condition.
% $V=V_1\cup V_2$ implies $V_{i+1}=V_{i-1}$, so we would need $v\in V_i$ implies $v$ is adjacent to two distinct elements of $V_{i+1\:mod\: 2}$.
In this particular case where $k=2$, for all $i$, $i+1\equiv i-1\: (mod\: 2)$ so each $v\in V_i$ is adjacent to two distinct vertices in $V_{i+1\: (mod\: 2)}$.
Observe that any cycling 2-factor will be a set of cycles with length a multiple of $k$.

%Introduction of the problem of interest
Let $G$ be a geometric graph $G\cong K_{kn}$, with vertices in convex position and $V=V_1\cup\ldots\cup V_k$ a labeled partition with $|V_i|=n$.
We are interested in studying the cycling 2-factors of $G$ that have the smallest number of intersections.
In addition, we are interested in finding the number of intersections of this cycling 2-factor.
In Figure \ref{figExamples}, we have a geometric graph $G\cong K_{6}$ with $k=3$, and two examples of cycling 2-factors of $G$.
Each element of the partition has two vertices indicated by different shades.
The first example, $(a)$, has two cycles with four intersections and the second, $(b)$, has one cycle with three intersections.
Later, we will show that the second example is in fact the unique cycling 2-factor with the minimum number of intersections for $G$.

%Example figure 1.
\begin{figure}[h]\label{figExamples}
	\centering
\begin{tikzpicture}

\definecolor{white}{rgb}{1.0,1.0,1.0}
\definecolor{black}{rgb}{0.0,0.0,0.0}
\definecolor{red}{rgb}{1.0,0.0,0.0}
\definecolor{blue}{rgb}{0.0,0.0,1.0}
\definecolor{green}{rgb}{0.0,1.0,0.0}
\definecolor{gray}{rgb}{0.8,0.8,0.8}
\tikzset{VertexStyle/.style = { shape=circle,
                                   fill         = none,
                                   minimum size = 2pt,
                                   text         = black,
                                   draw=none}}
\Vertex[LabelOut=false,L=\hbox{$\texttt{(a)}$},x=1.5cm,y=-.5cm]{u}
\tikzset{VertexStyle/.style = {shape        = circle,
                                   fill         = black,
                                   inner sep = 0pt,
                                   minimum size = 5pt,
                                   text         = black,
                                   draw}}
\Vertex[LabelOut=false,L=\hbox{$\texttt{}$},x=0.0cm,y=1.0cm]{v0}
\Vertex[LabelOut=false,L=\hbox{$\texttt{}$},x=.7cm,y=2.0cm]{v1}
\tikzset{VertexStyle/.style = {shape        = circle,
                                   fill         = gray,
                                   inner sep = 0pt,
                                   minimum size = 5pt,
                                   text         = black,
                                   draw}}
\Vertex[LabelOut=false,L=\hbox{$\texttt{}$},x=2.3cm,y=2.0cm]{v2}
\Vertex[LabelOut=false,L=\hbox{$\texttt{}$},x=3.0cm,y=1.0cm]{v3}
\tikzset{VertexStyle/.style = {shape        = circle,
                                   fill         = white,
                                   inner sep = 0pt,
                                   minimum size = 5pt,
                                   text         = black,
                                   draw}}
\Vertex[LabelOut=false,L=\hbox{$\texttt{}$},x=2.0cm,y=0.0cm]{v4}
\Vertex[LabelOut=false,L=\hbox{$\texttt{}$},x=1.0cm,y=0.0cm]{v5}
\Edge[lw=0.025cm,style={color=black,},](v1)(v3)
\Edge[lw=0.025cm,style={color=black,},](v0)(v2)
\Edge[lw=0.025cm,style={color=black,},](v3)(v4)
\Edge[lw=0.025cm,style={color=black,},](v2)(v5)
\Edge[lw=0.025cm,style={color=black,},](v5)(v0)
\Edge[lw=0.025cm,style={color=black,},](v4)(v1)

\end{tikzpicture}
\hspace{5mm}
\begin{tikzpicture}
\definecolor{white}{rgb}{1.0,1.0,1.0}
\definecolor{black}{rgb}{0.0,0.0,0.0}
\definecolor{red}{rgb}{1.0,0.0,0.0}
\definecolor{blue}{rgb}{0.0,0.0,1.0}
\definecolor{green}{rgb}{0.0,1.0,0.0}
\definecolor{gray}{rgb}{0.8,0.8,0.8}
\tikzset{VertexStyle/.style = { shape=circle,
                                   fill         = none,
                                   minimum size = 2pt,
                                   text         = black,
                                   draw=none}}
\Vertex[LabelOut=false,L=\hbox{$\texttt{(b)}$},x=1.5cm,y=-.5cm]{u}
\tikzset{VertexStyle/.style = {shape        = circle,
                                   fill         = black,
                                   inner sep = 0pt,
                                   minimum size = 5pt,
                                   text         = black,
                                   draw}}
\Vertex[LabelOut=false,L=\hbox{$\texttt{}$},x=0.0cm,y=1.0cm]{v0}
\Vertex[LabelOut=false,L=\hbox{$\texttt{}$},x=.7cm,y=2.0cm]{v1}
\tikzset{VertexStyle/.style = {shape        = circle,
                                   fill         = gray,
                                   inner sep = 0pt,
                                   minimum size = 5pt,
                                   text         = black,
                                   draw}}
\Vertex[LabelOut=false,L=\hbox{$\texttt{}$},x=2.3cm,y=2.0cm]{v2}
\Vertex[LabelOut=false,L=\hbox{$\texttt{}$},x=3.0cm,y=1.0cm]{v3}
\tikzset{VertexStyle/.style = {shape        = circle,
                                   fill         = white,
                                   inner sep = 0pt,
                                   minimum size = 5pt,
                                   text         = black,
                                   draw}}
\Vertex[LabelOut=false,L=\hbox{$\texttt{}$},x=2.0cm,y=0.0cm]{v4}
\Vertex[LabelOut=false,L=\hbox{$\texttt{}$},x=1.0cm,y=0.0cm]{v5}%
\Edge[lw=0.025cm,style={color=black,},](v1)(v2)
\Edge[lw=0.025cm,style={color=black,},](v0)(v3)
\Edge[lw=0.025cm,style={color=black,},](v3)(v4)
\Edge[lw=0.025cm,style={color=black,},](v2)(v5)
\Edge[lw=0.025cm,style={color=black,},](v5)(v0)
\Edge[lw=0.025cm,style={color=black,},](v4)(v1)

\end{tikzpicture}
\parbox{\capwidth}{\caption{$(a)$ has four intersections, $(b)$ has three and is the minimum.}}
\end{figure}
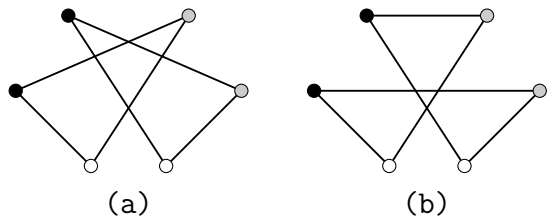

\section{Structural Results}
%Definition of the partition {u,v}+L_{uv}+R_{uv}.
\begin{definition}
Let $u$ and $v$ be distinct vertices in a geometric graph $G$.
The edge $uv\in E(G)$ induces a partition of $V(G)$ by $V(G)=\{u,v\}\cup L_{uv}\cup R_{uv}$, where $L_{uv}$ is a set of vertices in $G$ contained in an open half plane defined by the line connecting $u$ and $v$.
Similarly, $R_{uv}$ is the set of vertices in $G$ contained in the other open half plane defined by the line connecting $u$ and $v$.
\end{definition}

Observe that the half planes that produce sets $L_{uv}$ and $R_{uv}$ are distinguished only by the vertices of $G$, so the convention used for figures will be when viewing $uv$ as a vertical line with $u$ below $v$ then $L_{uv}$ will be on the {\em left} and $R_{uv}$ will be on the {\em right}. 

Given $u,v\in V(G)$, we will say {\em $u$ and $v$ are neighbors on the convex hull} if the line connecting $u$ and $v$ is a defining face of $conv(V(G))$, equivalently if $L_{uv}=\emptyset$ or $R_{uv}=\emptyset$.

\begin{definition}
(Following the notation from \cite{Diestel}.)
Let $A$ and $B$ be disjoint subsets of $V(G)$.
The set of edges in $G$ with a vertex in $A$ and a vertex in $B$ is denoted $E_G(A,B)$, and the number of such edges is denoted by $e_G(A,B)$.
\end{definition}

\begin{observation}\label{obs1}
The edges $uv$ and $st$ intersect if and only if $s\in L_{uv}$ and $t\in R_{uv}$, or the reverse.
Thus, the number of edges that intersect $uv$ in any $H\subseteq G$ will be $e_H(L_{uv},R_{uv})$.
\end{observation}

This tells us that the exact coordinates of the points are less important than the sets $L_{uv}$ and $R_{uv}$ that contain them.
Thus, we have a combinatorial framework to consider the intersections.

%Neccessary condition
\begin{theorem}\label{thmParallel}
Given a cycling 2-factor $H$ of $G=(V,E)$, $V=V_1\cup\ldots\cup V_k$, $k\geq 3$.
If $H$ has the minimum number of intersections over all cycling 2-factors of $G$, then any two edges $vu,rs\in E(H)$ such that $u,r\in V_i$ and $v,s\in V_{i+1\: (mod\: k)}$ will not intersect.
\end{theorem}

\begin{proof}
Given $G$ as above, $H$ a cycling two factor of $G$, and $u,r\in V_i$, $v,s\in V_{i+1\: (mod\: k)}$ such that $vu,rs\in E(H)$ intersect.

These two edges produce four sets of vertices $S_1=L_{uv}\cap L_{rs},S_2=L_{uv}\cap S_{rs},S_3=R_{uv}\cap L_{uv}$ and $S_4=R_{uv}\cap R_{st}$.
See Graph $(H)$ in Figure \ref{figParallel}.
Observe that the number of intersections of edges that do not include $u,v,r$ or $s$ for this pair of edges is
\[e(S_1,S_2)+e(S_1,S_3)+2e(S_1,S_4)+e(S_2,S_4)+2e(S_2,S_3)+e(S_3,S_4).\]

Let $H'$ be $H$ with edges $us$ and $rv$ and without edges $uv$ and $rs$.
See graph $H'$ in Figure \ref{figParallel}.

Observe that the two new edges do not intersect.
Now, the number of edges that do not include $u,v,r$ or $s$ which intersect the new edges is
\[e(S_1,S_2)+e(S_1,S_3)+2e(S_1,S_4)+e(S_2,S_4)+e(S_3,S_4).\]
Thus the number of intersections of edges that do not include $u,v,r$ or $s$ for this pair of edges does not increase.

Suppose an edge intersects $us$ and has $v$ or $r$ as an endpoint.
Then the other vertex must be in $S_1$, since $u\in R_{rs}$ and $r\in R_{uv}$ we conclude that this edge would have intersected at least one of $uv$ or $rs$ in $H$.
A similar statement holds for any edge that intersects $vr$ and has $u$ or $s$ as a vertex.
Since there is at least one fewer intersection, it follows that $H'$ has fewer intersections than $H$.

This concludes the proof.
\end{proof}

\begin{figure}[h]\label{figParallel}
	\centering
\begin{tikzpicture}
\draw (1,1) circle (1.5cm);
\draw (2.5,1.7) arc (28:-28:1.6cm);
\draw (-.5,1.7) arc (-28:28:-1.55cm);
\draw (.3,2.5) arc (120:60:1.5cm);
\draw (.3,-.5) arc (-120:-60:1.5cm);
\definecolor{white}{rgb}{1.0,1.0,1.0}
\definecolor{black}{rgb}{0.0,0.0,0.0}
\definecolor{red}{rgb}{1.0,0.0,0.0}
\definecolor{blue}{rgb}{0.0,0.0,1.0}
\definecolor{gray}{rgb}{0.5,0.5,0.5}
\tikzset{VertexStyle/.style = { shape=rectangle,
                                   fill         = none,
                                   font = \footnotesize,
                                   minimum size = 10pt,
                                   text         = black,
                                   draw=none}}
\Vertex[LabelOut=false,L=\hbox{$\texttt{$(H)$}$},x=-.20cm,y=-1.0cm]{u}
\Vertex[LabelOut=false,L=\hbox{$\texttt{$S_2$}$},x=-1cm,y=1.0cm]{u0}
\Vertex[LabelOut=false,L=\hbox{$\texttt{$S_3$}$},x=3cm,y=1.0cm]{u1}
\Vertex[LabelOut=false,L=\hbox{$\texttt{$S_1$}$},x=1cm,y=3cm]{u2}
\Vertex[LabelOut=false,L=\hbox{$\texttt{$S_4$}$},x=1cm,y=-1cm]{u3}
\tikzset{VertexStyle/.style = {shape        = circle,
                                   fill         = white,
                                   font = \scriptsize,
                                   minimum size = 5pt,
                                   text         = black,
                                   draw}}
\Vertex[LabelOut=false,L=\hbox{$\texttt{r}$},x=0.0cm,y=0.0cm]{v0}
\Vertex[LabelOut=false,L=\hbox{$\texttt{u}$},x=0.0cm,y=2cm]{v1}
\tikzset{VertexStyle/.style = {shape        = circle,
                                   fill         = white,
                                   font = \scriptsize,
                                   minimum size = 5pt,
                                   text         = black,
                                   draw}}
\Vertex[LabelOut=false,L=\hbox{$\texttt{v}$},x=2cm,y=0.0cm]{v2}
\Vertex[LabelOut=false,L=\hbox{$\texttt{s}$},x=2cm,y=2cm]{v3}
\Edge[lw=0.025cm,style={color=black,},](v0)(v3)
\Edge[lw=0.025cm,style={color=black,},](v1)(v2)
\end{tikzpicture}
\hspace{5mm}
\begin{tikzpicture}
\draw (1,1) circle (1.5cm);
\draw (2.5,1.7) arc (28:-28:1.6cm);
\draw (-.5,1.7) arc (-28:28:-1.55cm);
\draw (.3,2.5) arc (120:60:1.5cm);
\draw (.3,-.5) arc (-120:-60:1.5cm);
\definecolor{white}{rgb}{1.0,1.0,1.0}
\definecolor{black}{rgb}{0.0,0.0,0.0}
\definecolor{red}{rgb}{1.0,0.0,0.0}
\definecolor{blue}{rgb}{0.0,0.0,1.0}
\definecolor{gray}{rgb}{0.5,0.5,0.5}
\tikzset{VertexStyle/.style = { shape=rectangle,
                                   fill         = none,
                                   font = \footnotesize,
                                   minimum size = 10pt,
                                   text         = black,
                                   draw=none}}
\Vertex[LabelOut=false,L=\hbox{$\texttt{$(H')$}$},x=-.20cm,y=-1.0cm]{u}
\Vertex[LabelOut=false,L=\hbox{$\texttt{$S_2$}$},x=-1cm,y=1.0cm]{u0}
\Vertex[LabelOut=false,L=\hbox{$\texttt{$S_3$}$},x=3cm,y=1.0cm]{u1}
\Vertex[LabelOut=false,L=\hbox{$\texttt{$S_1$}$},x=1cm,y=3cm]{u2}
\Vertex[LabelOut=false,L=\hbox{$\texttt{$S_4$}$},x=1cm,y=-1cm]{u3}
\tikzset{VertexStyle/.style = {shape        = circle,
                                   fill         = white,
                                   font = \scriptsize,
                                   minimum size = 5pt,
                                   text         = black,
                                   draw}}
\Vertex[LabelOut=false,L=\hbox{$\texttt{r}$},x=0.0cm,y=0.0cm]{v0}
\Vertex[LabelOut=false,L=\hbox{$\texttt{u}$},x=0.0cm,y=2cm]{v1}
\tikzset{VertexStyle/.style = {shape        = circle,
                                   fill         = white,
                                   font = \scriptsize,
                                   minimum size = 5pt,
                                   text         = black,
                                   draw}}
\Vertex[LabelOut=false,L=\hbox{$\texttt{v}$},x=2cm,y=0.0cm]{v2}
\Vertex[LabelOut=false,L=\hbox{$\texttt{s}$},x=2cm,y=2cm]{v3}
\Edge[lw=0.025cm,style={color=black,},](v0)(v2)
\Edge[lw=0.025cm,style={color=black,},](v1)(v3)
\end{tikzpicture}
	\parbox{\capwidth}{\caption{Regions produced by intersecting edges and non intersecting edges. Where $u,r\in V_i$ and $v,s\in V_{i+1}$.}}
\end{figure}
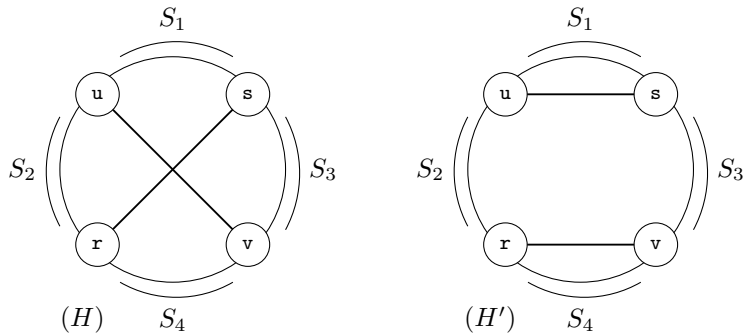

%Lemma about outside edges
\begin{lemma}\label{lemOutsideEdge}
Let $G=(V,E)$, where $V=V_1\cup\ldots\cup V_k$, and $k\geq 3$. Let $H$ be a cycling 2-factor of $G$.
Suppose $u\in V_j$ and $v\in V_{j+1\: mod\: k}$ so that $u$ and $v$ are neighbors on the convex hull, but $uv\notin H$, then there is a cycling 2-factor $H'$ that contains $uv$ and has at most two more intersections than $H$.
\end{lemma}

\begin{proof}
Let $G$, $H$, $u$ and $v$ be as in the statement.
Let $r\in V_{j+1\: (mod\: k)}$ and $s\in V_{j}$ so that $ur,vs\in E(H)$.

Let $H'$ be $H$ with edges $ur$ and $vs$ and without edges $uv$ and $rs$.
Observe that $H'$ is also a cycling 2-factor as the only vertices with different adjacency are $u,v,r$ and $s$ and they still satisfy the cycling condition.

Since $u$ and $v$ are neighbors on the convex hull, without loss of generality, we assume $L_{vu}=\emptyset$.

Assume $ur$ and $vs$ intersect in $H$.
By Theorem \ref{thmParallel} we conclude that the number of intersections of $H'$ is at least one less than the number of intersections of $H$.

Now, assume $ur$ and $vs$ do not intersect in $H$.
See Graph $(H)$ in Figure \ref{figLemma}.
Let $v,s\in R_{ur}$ and $u,r\in L_{vs}$.
This gives a partition of $V\setminus\{u,v,r,s\}$ into $L_{ur}\cup L_{rs}\cup L_{sv}$, (see Figure \ref{figLemma}).
Moreover, since $L_{vu}=\emptyset$, up to relabeling, we get that the sets partition the remaining vertices.
Thus the number of edges that intersect $rs$ can be bounded by,
\begin{align*}
e(L_{rs},R_{rs})&=e(L_{rs},L_{ur})+e(L_{rs},L_{sv})\\
& \leq e(L_{ur},R_{ur})+e(R_{sv},L_{sv}).
\end{align*}
This follows as $e(L_{rs},L_{ur})$ is less than the total number of edges that intersect $ur$ which is $e(L_{ur},R_{ur})$, similarly for $sv$.
This implies that the only way to add new intersections is from edges that contain $u,v,r$ or $s$.
Since $uv$ can not intersect any edges, the only other possibility is $rs$ intersects edges that contain $u$ or $v$.
However, there are at most two such edges, hence there are at most two new intersections. 

Thus, by construction, there is a cycling 2-factor $H'$ with at most two more intersections than $H$ that contains edge $uv$.
\end{proof}

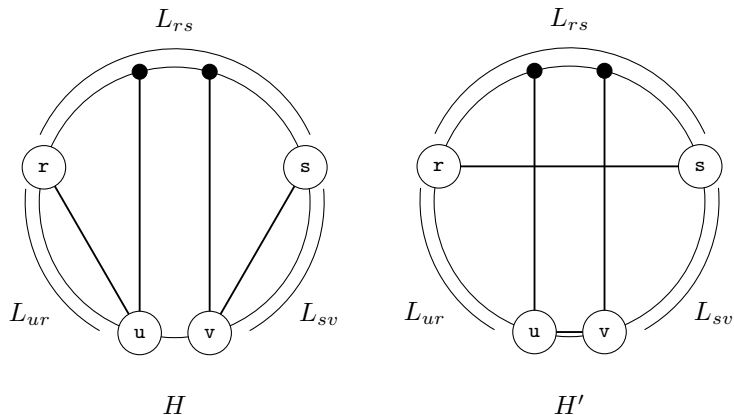
\begin{figure}[width=1.00\textwidth,h]\label{figLemma}
	\centering
	%\captionsetup{width=0.7\textwidth}
\begin{tikzpicture}[scale=.9]
\definecolor{white}{rgb}{1.0,1.0,1.0}
\definecolor{black}{rgb}{0.0,0.0,0.0}
\definecolor{red}{rgb}{1.0,0.0,0.0}
\definecolor{blue}{rgb}{0.0,0.0,1.0}
\definecolor{gray}{rgb}{0.5,0.5,0.5}
\draw (0,0) circle (2cm);
\draw (2,1) arc (25:155:2.2cm);
\draw (2.2,.2) arc (5:-60:2.2cm);
\draw (-2.2,.2) arc (175:240:2.2cm);
\tikzset{VertexStyle/.style = { shape=rectangle,
                                   fill         = none,
                                   minimum size = 10pt,
                                   font = \footnotesize,
                                   text         = black,
                                   draw=none}}
\Vertex[LabelOut=false,L=\hbox{$\texttt{$H$}$},x=0cm,y=-3.0cm]{u}                                  
\Vertex[LabelOut=false,L=\hbox{$\texttt{$L_{ur}$}$},x=-2.138cm,y=-1.647cm]{u0}
\Vertex[LabelOut=false,L=\hbox{$\texttt{$L_{rs}$}$},x=0.0cm,y=2.7cm]{u1}
\Vertex[LabelOut=false,L=\hbox{$\texttt{$L_{sv}$}$},x=2.138cm,y=-1.647cm]{u2}
\tikzset{VertexStyle/.style = {shape        = circle,
                                   fill         = white,
                                   font = \scriptsize,
                                   minimum size = 5pt,
                                   text         = black,
                                   draw}}
\Vertex[LabelOut=false,L=\hbox{$\texttt{v}$},x=.517cm,y=-1.931cm]{v0}
\Vertex[LabelOut=false,L=\hbox{$\texttt{r}$},x=-1.931cm,y=.518cm]{v1}
\Vertex[LabelOut=false,L=\hbox{$\texttt{u}$},x=-.517cm,y=-1.931cm]{v2}
\Vertex[LabelOut=false,L=\hbox{$\texttt{s}$},x=1.931cm,y=.518cm]{v3}
\tikzset{VertexStyle/.style = {shape        = circle,
                                   fill         = black,
                                   font = \scriptsize,
                                   inner sep = 2pt,
                                   minimum size = 5pt,
                                   text         = black,
                                   draw}}
\Vertex[LabelOut=false,L=\hbox{$\texttt{}$},x=-.517cm,y=1.931cm]{v4}
\Vertex[LabelOut=false,L=\hbox{$\texttt{}$},x=.517cm,y=1.931cm]{v5}
%
%\Edge[lw=0.025cm,style={color=black,},](v0)(v2)
%\Edge[lw=0.025cm,style={color=black,},](v1)(v3)
\Edge[lw=0.025cm,style={color=black,},](v0)(v3)
\Edge[lw=0.025cm,style={color=black,},](v1)(v2)
\Edge[lw=0.025cm,style={color=black,},](v0)(v5)
\Edge[lw=0.025cm,style={color=black,},](v2)(v4)
\end{tikzpicture}
\hspace{3mm}
\begin{tikzpicture}[scale=.9]
\definecolor{white}{rgb}{1.0,1.0,1.0}
\definecolor{black}{rgb}{0.0,0.0,0.0}
\definecolor{red}{rgb}{1.0,0.0,0.0}
\definecolor{blue}{rgb}{0.0,0.0,1.0}
\definecolor{gray}{rgb}{0.5,0.5,0.5}
\draw (0,0) circle (2cm);
\draw (2,1) arc (25:155:2.2cm);
\draw (2.2,.2) arc (5:-60:2.2cm);
\draw (-2.2,.2) arc (175:240:2.2cm);
\tikzset{VertexStyle/.style = { shape=rectangle,
                                   fill         = none,
                                   minimum size = 10pt,
                                   font = \footnotesize,
                                   text         = black,
                                   draw=none}}
\Vertex[LabelOut=false,L=\hbox{$\texttt{$H'$}$},x=0cm,y=-3.0cm]{u}                                  
\Vertex[LabelOut=false,L=\hbox{$\texttt{$L_{ur}$}$},x=-2.138cm,y=-1.647cm]{u0}
\Vertex[LabelOut=false,L=\hbox{$\texttt{$L_{rs}$}$},x=0.0cm,y=2.7cm]{u1}
\Vertex[LabelOut=false,L=\hbox{$\texttt{$L_{sv}$}$},x=2.138cm,y=-1.647cm]{u2}
\tikzset{VertexStyle/.style = {shape        = circle,
                                   fill         = white,
                                   font = \scriptsize,
                                   minimum size = 5pt,
                                   text         = black,
                                   draw}}
\Vertex[LabelOut=false,L=\hbox{$\texttt{v}$},x=.517cm,y=-1.931cm]{v0}
\Vertex[LabelOut=false,L=\hbox{$\texttt{r}$},x=-1.931cm,y=.518cm]{v1}
\Vertex[LabelOut=false,L=\hbox{$\texttt{u}$},x=-.517cm,y=-1.931cm]{v2}
\Vertex[LabelOut=false,L=\hbox{$\texttt{s}$},x=1.931cm,y=.518cm]{v3}
\tikzset{VertexStyle/.style = {shape        = circle,
                                   fill         = black,
                                   font = \scriptsize,
                                   inner sep = 2pt,
                                   minimum size = 5pt,
                                   text         = black,
                                   draw}}
\Vertex[LabelOut=false,L=\hbox{$\texttt{}$},x=-.517cm,y=1.931cm]{v4}
\Vertex[LabelOut=false,L=\hbox{$\texttt{}$},x=.517cm,y=1.931cm]{v5}
\Edge[lw=0.025cm,style={color=black,},](v0)(v2)
\Edge[lw=0.025cm,style={color=black,},](v1)(v3)
%\Edge[lw=0.025cm,style={color=black,},](v0)(v3)
%\Edge[lw=0.025cm,style={color=black,},](v1)(v2)
\Edge[lw=0.025cm,style={color=black,},](v0)(v5)
\Edge[lw=0.025cm,style={color=black,},](v2)(v4)
\end{tikzpicture}
	\parbox{\capwidth}{\caption{Replacing $ur$ and $sv$ with $uv$ and $rs$ adds at most two more intersections.}}
\end{figure}

%Lemma about no intersecting 2-factors
\begin{lemma}\label{lemSeparatingCycles}
Let $G=(V,E)$, with $V=V_1\cup\ldots\cup V_k$, and $k\geq 3$, and suppose $G$ contains a cycling 2-factor with no intersections.
Then $G$ has a cycling 2-factor with no intersections consisting only of $k$-cycles.
\end{lemma}

\begin{proof}
Let $G$ be as above, and suppose $G$ contains a cycling 2-factor with no intersections.
Let $H$ be a cycling 2-factor with with no intersections.
Suppose that $H$ is the 2-factor with the smallest maximum length cycle and the smallest number of the longest cycles.
Observe that since $H$ has no intersections, each cycle is disjoint.
Thus, we can study each cycle separately.

Suppose $C$ is a cycle length greater than $k$.
Let $x_1,\ldots,x_k$ be a path of $C$, and let $x_0$ be adjacent to $x_1$ and $x_k$ adjacent to $x_{k+1}$ in $C$.
Let $H'$ be $H$ with edges $x_1x_k$ and $x_0x_{k+1}$ and without edges $x_0x_1$ and $x_k x_{k+1}$.
This produces two cycles with length smaller than the length of $C$.

As this contradicts the choice of $H$, we conclude that $H$ did not have any cycles of length greater than $k$.
\end{proof}

\section{Upper Bounds}
We define a {\em transposition} $(u,v)$ of a graph embedding $G$ to be an embedding $\hat{G}$ that changes the embedding of $G$ by switching the coordinates of $u$ and $v$.
When performing this switch with a cycling 2-factor, $H$, we will assume that the new cycling 2-factor has the same edges.
That is, adjacent vertices do not change in $H$, just the embedding of $G$.

\begin{definition}
Let $f(G)$ be the minimum number of transpositions $(u,v)$, where $u$ and $v$ are neighbors on the convex hull, required to convert $G$ into an embedding with a cycling 2-factor with no intersections.
\end{definition}

%Upper bound theorem
\begin{theorem}\label{thmUpperBound}
Let $G=(V,E)$, where $V=V_1\cup\ldots\cup V_k$, and $k\geq 3$, and let $H$ be a cycling 2-factor of $G$ with the minimum number of intersections.
Then, the number of intersections of $H$ is bounded above by $4 f(G)$.
Moreover, for $k=3$, the number of intersections of $H$ is bounded above by $3 f(G)$.
\end{theorem}

\begin{proof}
Observe that the transposition $(u,v)$ for $u$ and $v$ neighbors on the convex hull will not change $L_{rs}$ or $R_{rs}$ for all edges $rs$ disjoint from $\{u,v\}$.
Thus, this operation will not add any intersections between edges that do not contain $u$ or $v$ as an end point.
There are at most four edges containing $u$ or $v$ thus, each transposition adds at most four intersections.
Starting with a cycling 2-factor with no intersections, apply the transpositions in reverse order adding at most four intersections per transposition to get a cycling 2-factor of $G$ with at most $4f(G)$ intersections.
Thus, the minimum number of intersections is bounded above by $4f(G)$.

Now, assume $k=3$.
Observe that if $(u,v)$ is a transposition in the minimum sequence, then $u$ and $v$ would not be in the same part of $V$.
Assume that the minimum sequence of transpositions includes $(u,v)$.
Note that $u\in V_i$, $v\in V_j$ implies $i=j\pm1\:(mod\: 3)$.
Hence, by Lemma \ref{lemOutsideEdge}, there is a cycling 2-factor, $H'$, with at most two more intersections than the current cycling 2-factor that uses $uv$.
In $H'$, $u$ and $v$ have two edges that could intersect adding at most one more intersection when $(u,v)$ is applied to $H'$.
Starting with a cycling 2-factor with no intersections and apply the transpositions in reverse order to the cycling 2-factor that is produced by Lemma \ref{lemOutsideEdge}. 
Thus, the minimum number of intersections is bounded above by $3f(G)$ when $k=3$.
\end{proof}

%Theorem about tightness.
\begin{theorem}\label{thmNumericalBound}
Let $G=(V,E)$, $V=V_1\cup V_2\cup V_3$, and let $H$ be a cycling 2-factor of $G$.
Then the number of intersections of $H$ is bounded above by $\frac{3n(n-1)}{2}$ where $|V_1|=n$.
\end{theorem}

\begin{proof}
Let $G$ be as in the statement.
Proof by induction on $|V_1|=n$.
We will construct a sequence of transpositions $(u,v)$ of neighbors on the convex hull of length at most $\frac{n(n-1)}{2}$.

If $|V_1|=1$ then $G\cong C_3$, and the only cycling 2-factor of $G$ has no intersections.
Now, assume that for any graph $G$ with $|V_1|=n$, there is a cycling 2-factor with at most $\frac{3n(n-1)}{2}$ intersections.
Assume $G$ has $|V_1|=n+1$.
Observe that there are vertices $u_1\in V_1$ and $u_2\in V_2$ such that $u_1$ and $u_2$ are neighbors on the convex hull.
Let $u_3\in V_3$ be the closest vertex in cyclic order along the convex hull, let the transpositions $(u_3,t_1),(u_3,t_2),\ldots,(u_3,t_i)$ give an embedding where $u_3$ is a neighbor of either $u_1$ or $u_2$.
Observe that at most $n$ transpositions will be needed.

Using the inductive hypothesis, we construct a cycling 2-factor on $G\setminus\{u_1,u_2,u_3\}$ with at most $\frac{3n(n-1)}{2}$ intersections.
Now, write $G$ with $u_1,u_2$ and $u_3$ as neighbors on the convex hull, with $u_1$ and $u_2$ in their original position in $G$.
Apply the transpositions $(u_3,t_i),(u_3,t_{i-1}),\ldots,(u_3,t_1)$, in order.
By Lemma \ref{lemOutsideEdge} and the method used in Theorem \ref{thmUpperBound} we can construct a cycling 2-factor of $G$ with at most $\frac{3n(n-1)}{2}+3n=\frac{3(n+1)n}{2}$ intersections as desired.

Thus, we have an upper bound on the intersections of $H$ of $\frac{3n(n-1)}{2}$ where $|V_1|=n$.
\end{proof}

\begin{corollary}
For each $n$ there is a unique embedding of $G\cong K_{3n}$, $V=V_1\cup V_2\cup V_3$, $|V_i|=n$, $i=1,2,3$ where the minimum number of intersections of a cycling two factor is $\frac{3n(n-1)}{2}$.
\end{corollary}
\begin{proof}
Let $V=\{v_1,\ldots,v_{3n}\}$ be in cyclic order.
Set $V_1=\{v_1,\ldots,v_n\}$, $V_2=\{v_{n+1},\ldots,v_{2n}\}$, and $V_3=\{v_{2n+1},\ldots,v_{3n}\}$.
There is a unique set of edges that satisfy the condition in Theorem \ref{thmParallel}, and, thus, must have the minimal number of intersections.
Namely, $v_{i}$ is adjacent to $v_{2n+1-i}$ and $v_{3n+1-i}$ for $i=1,\ldots,n$.
The other edge set has $v_{n+i}$ adjacent to $v_{3n+1-i}$ for $i=1,\ldots,n$.
In Figure \ref{figExamples}, Graph $(b)$ is the case when $n=2$.

Given an edge $uv$ in this cycling 2-factor, where $u\in V_1$ and $v\in V_2$, the only edges in $E(V_1,V_3)$ that intersect this edge will have a vertex in $V_1$ with higher index then $u$ in the cyclic order, the only edges in $E(V_2,V_3)$ that intersect this edge have a vertex in $V_2$ with smaller index then $v$ in the cyclic order.
This gives a total of $n(n-1)$ intersections of edges in $E(V_1,V_2)$ with edges in $E(V_1,V_3)\cup E(V_2,V_3)$.
Finally, given an edge $uv$ in the cycling 2-factor where $u\in V_2$ and $v\in V_3$ the edges in $E(V_1,V_3)$ that intersect this edge will have a vertex in $V_3$ with smaller index than $v$.
This adds an additional $\frac{n(n-1)}{2}$ intersections between $E(V_2,V_3)$ and $E(V_1,V_3)$.
This gives a total of $\frac{3n(n-1)}{2}$ intersections with $|V_1|=n$ and so the bound given in Theorem \ref{thmNumericalBound} is tight on this graph.

To see uniqueness it suffices to show that there is a unique graph where it requires a minimum of $n-1$ transpositions to get three vertices, $u_1\in V_1$, $u_2\in V_2$ and $u_3\in V_3$, to be neighbors on the convex hull.

Suppose that $u_1\in V_1$ and $u_2\in V_2$ are neighbors on the convex hull, and assume that it would require at least $n-1$ transpositions to get a vertex in $V_3$ adjacent to either on the convex hull.
Then, there are maximal sets of consecutive, with respect to the convex hull, vertices from $V_1\cup V_2$ one that contains $u_1$, and not $u_2$ and another that contains $u_2$ and not $u_1$.
Each set has $n$ vertices, otherwise there is a vertex in $V_3$ of distance less than $n-1$ in cyclic order from $u_1$ or $u_2$.
Thus, without loss of generality we may assume $v_1,v_2,\ldots,v_n=u_1,v_{n+1}=u_2,\ldots,v_n$ are in $V_1\cup V_2$, and $\{v_{2n+1},\ldots,v_{3n}\}=V_3$.
If there is $v_i\in\{v_1,\ldots,v_{n-1}\}$ then there is $v_j\in\{v_1,\ldots,v_{n-1}\}$ such that $v_j\in V_2$ and $v_{j+1}\in V_1$ where it would require $j-1<n-1$ transpositions to get $v_{3n}$ adjacent to $v_j$.
Thus, $\{v_1,\ldots,v_n\}=V_1$ and $\{v_{n+1},\ldots,v_{2n}\}=V_2$.

This proves the uniqueness of the graph.
\end{proof}

\section*{Acknowledgements}

The author wishes to thank Michael Burr, Matthew Macauley, and Marc Lipman for their help with this paper.
Without their support this paper would not have been possible.

\end{document}